\documentclass[11pt]{amsart}

\usepackage[colorlinks=true, linkcolor=blue, anchorcolor=blue, citecolor=blue, filecolor=blue, menucolor= blue, pagecolor=black, urlcolor=red]{hyperref}

\usepackage{amssymb,latexsym,amscd}

\theoremstyle{plain}
\newtheorem{thm}{Theorem}[section]
\newtheorem{prop}[thm]{Proposition}
\newtheorem{lem}[thm]{Lemma}

\theoremstyle{remark}
\newtheorem{rem}[thm]{Remark}

\newcommand{\C}[1]{\hbox{$\mathcal #1$}}
\newcommand{\pr}{\hbox{${\bf P}$}}

\long\def\eatit#1{}

\begin{document}

\title{Configuration types and cubic surfaces}

\author{Elena Guardo and Brian Harbourne}

\address{Elena Guardo\\
Dipartimento di Matematica e Informatica\\
Viale A. Doria 6, 95100\\
Catania, Italy\\
\url{http://www.dmi.unict.it/~guardo}}
\email{guardo@dmi.unict.it}

\address{Brian Harbourne\\
Department of Mathematics\\
University of Nebraska-Lincoln\\
Lincoln, NE 68588-0130\\
USA\\
\url{http://www.math.unl.edu/~bharbour/}}
\email{bharbour@math.unl.edu}

\date{May 3, 2008}

\begin{abstract}
This paper is a sequel to the paper \cite{refGH}.
We relate the matroid notion of a combinatorial geometry
to a generalization which we call a configuration type.
Configuration types arise when one classifies the
Hilbert functions and graded Betti numbers for
fat point subschemes supported at $n\le8$ essentially distinct points of the
projective plane. Each type gives rise to a surface $X$ obtained by blowing up the points.
We classify those types such that $n=6$ and $-K_X$ is nef. The surfaces obtained
are precisely  the desingularizations of the normal cubic surfaces.
By classifying configuration types we recover in all characteristics the
classification of normal cubic surfaces, which is well-known
in characteristic 0 \cite{refBW}.
As an application of our classification of configuration types,
we obtain a numerical procedure for determining
the Hilbert function and graded Betti numbers for the ideal
of any fat point subscheme  $Z=m_1p_1+\cdots+m_6p_6$
such that the points $p_i$ are essentially distinct and $-K_X$ is nef,
given only the configuration type of the points
$p_1,\ldots,p_6$ and the coefficients $m_i$.
\end{abstract}

\thanks{Acknowledgments: We would like to thank the GNSAGA, the NSA, the NSF and the
Department of Mathematics at UNL for their support of the authors' research and
of E. Guardo's visits in 2003, 2004 and 2005 while this work was carried out.}



\maketitle

\section{Introduction}\label{intro}

\subsection*{Matroids and Combinatorial Geometries}
The combinatorial classification of points in projective space
leads one to the concept of combinatorial geometries.
Intuitively, a combinatorial geometry of rank $N+1$ or less and size $n$
is an abstract specification
of linear dependencies among a set of $n$ points spanning a space of dimension at most $N$.
Formally, a \emph{combinatorial geometry\/} is a matroid
without loops or parallel elements \cite{refHandbook}.
We are interested in the case of $n$ points in the projective plane,
in which case one can regard a combinatorial geometry
as just being a collection of subsets of the set $\{1,\dots,n\}$,
where each subset has at least two elements and
two subsets which have two elements in common
must be equal. We say that a given combinatorial geometry
is \emph{representable\/} over a field $k$ (in this paper the field $k$ will
be assumed to be algebraically closed, but not necessarily of characteristic 0)
if there is a collection of distinct points $p_1,\ldots,p_n\in \pr^2_k$ such that
a subset $\{i_1,\ldots,i_r\}\subseteq \{1,\ldots,n\}$ is an element of the combinatorial
geometry if and only if $\{p_{i_1},\ldots,p_{i_r}\}$ is a maximal
collinear subset of $\{p_1,\ldots,p_n\}$.

\subsection*{Combinatorial Geometries as Matrices}
We can think of a combinatorial geometry
of rank up to 3 and size $n$ with $g$ elements
as specifying a correspondence between a set of $g$
lines and $n$ points, where each line is defined by a maximal collinear subset
of the points. If we enumerate the lines, then specifying
the combinatorial geometry is equivalent to giving a 0-1 matrix $M$
where the entry in row $i$ and column $j$ is a 1 if and only if
the $i$th line contains the $j$th point. (Although any two points
determine a line, it is convenient to ignore
any row with exactly two 1's, and so this is what we will do.
This does no harm, since if we know all maximal collinear subsets
containing more than two points, we can recover all of
the maximal two point subsets. Since there typically are a lot of
maximal two point subsets, it is impractical to include them in the
matrix $M$.)

Thus we can regard combinatorial geometries on $n$ points in the plane
as being matrices $M$ with $n$ columns, such that each entry
of each row is either a 0 or a $1$, the sum of the entries
for a row is always at least $3$, and the dot product
of two different rows is either 0 or $1$. (If no three points
are collinear, then the matrix would have no rows.)

\subsection*{An Algebraic Geometric Perspective on Combinatorial Geometries}
A combinatorial geometry can be regarded as
telling us when more than the expected number
of points are to lie on a given line, the expected number being 2.
But in algebraic geometry we are also interested in
the possibility of points being special with respect
to curves of higher (and lower) degrees. Thus 6 points
on a conic are special (in which case the points
are special with respect to a curve of degree 2).
Similarly, it is special to have one point be infinitely near another
(in which case we can regard the points as being special
with respect to a curve of degree 0).

For the purposes of this paper, we now want to introduce
an alternative, algebraic geometric, approach
to combinatorial geometries, which we will then generalize
in the case of $n=6$ points
to sets of points being special with respect to curves other than lines.
Consider a set of distinct points $p_1,\ldots,p_n\in \pr^2_k$ which represents
a given combinatorial geometry, and let $\pi:X\to \pr^2_k$
be the morphism obtained by blowing up the points.
The divisor class group $\hbox{Cl}(X)$ is the free abelian
group on the divisors classes $L$, $E_1,\ldots,E_n$, where $L$ is the pullback
to $X$ of the class of a line on $\pr^2_k$ and
$E_i$ is the class of $\pi^{-1}(p_i)$.
The group $\hbox{Cl}(X)$ supports an intersection form;
it is the bilinear form
defined by requiring that the classes
$L$, $E_1,\ldots,E_n$ be orthogonal with $L^2=1$
and $E_i^2=-1$ for all $i$.

To ignore trivial cases, assume that $n>1$.
If the points $p_{i_1},\dots,p_{i_r}$, with $r>1$,
are collinear and none of the other points $p_i$ are on the same line,
then the class of the proper transform of the
line through those points is $C=L-E_{i_1}-\cdots-E_{i_r}$,
so the classes of the proper transforms of the lines
corresponding to the elements of the combinatorial geometry
are precisely the classes of all prime divisors $C$ on $X$
with $C^2<0$ and $C\cdot L=1$. If we construct a matrix $M'$ by
changing the sign of each nonzero entry of $M$ and then
prepending a column of 1's to the left side of $M$,
we obtain a matrix $M'$ whose rows specify (in terms of the basis
$L$, $E_1,\ldots,E_n$) the set of
all classes of prime divisors $C$ on $X$ such that
$C^2<0$ and $C\cdot L=1$ (with classes having $C^2=-1$
suppressed, corresponding to suppression in $M$ of rows
with exactly two 1's in them).

For example, given 5 points $p_1,\ldots,p_5$ such that
the maximal collinear subsets (ignoring two point subsets)
are points 1, 4 and 5 and points 2, 3 and 4, the matrices $M$
and $M'$ are:

$$\hbox{
$M=\left(\begin{matrix}
1 & 0 & 0 & 1 & 1\\
0 & 1 & 1 & 1 & 0\\
\end{matrix}\right)$
\hbox to.5in{\hfil}
$M'=\left(\begin{matrix}
1 & -1 & 0 & 0 & -1 & -1\\
1 & 0 & -1 & -1 & -1 & 0\\
\end{matrix}\right)$.}$$

The rows of $M'$ specify the classes
$L-E_1-E_4-E_5$ and $L-E_2-E_3-E_4$.

Thus we now can regard combinatorial geometries on $n$ points in the plane
as being matrices $M'$ with $n+1$ columns, where the first entry in each row is
a 1, each remaining entry is either a 0 or a $-1$, the sum of the entries of each row
(i.e., the intersection product of a row with itself,
with respect to the bilinear form defined above)
is always at most $-2$, and the intersection product
of two different rows is either 1 or 0.

\subsection*{Infinitely near points, Blow ups, the Intersection form and Exceptional configurations}
We now recall the notion of points being infinitely near.
Let $\pi: X\to \pr^2$ be the morphism
obtained as a sequence of blow ups of points in the following way.
Let $p_1\in X_0=\pr^2$,
and let $p_2\in X_1$, $\ldots$, $p_n\in X_{n-1}$, where, for
$0\le i\le n-1$,
$\pi_{i+1}:X_{i+1}\to X_i$ is the blow up of $p_{i+1}$. We will
denote $X_n$ by $X$
and the composition $X=X_n\to\cdots\to X_0=\pr^2$ by $\pi$. We say that the points
$p_1,\ldots,p_n$ are \emph{essentially distinct\/} points of $\pr^2$ \cite{refFreeRes};
note for $j>i$ that we may have $\pi_i\circ\cdots\circ\pi_{j-1}(p_j)=p_i$,
in which case we say $p_j$ is infinitely near $p_i$.
(If no point is infinitely near another, the points are just distinct points of $\pr^2$
and $X$ is just the surface obtained by blowing
the points up in a particular order, but the order does not matter.
If the points are only essentially distinct, then $p_i$
needs to be blown up before $p_j$ whenever $p_j$ is infinitely near
$p_i$.)

We denote by $E_i$ the class of the
1-dimensional scheme-theoretic fiber
of $X=X_n\to X_{i-1}$  over $p_i$ and the pullback to $X$ of the
class of a line in $\pr^2$ by $L$. As before, the classes
$L,E_1,\ldots,E_n$ form a basis over the integers
of the divisor class group
$\hbox{Cl}(X)$, which is a free abelian group of rank
$n+1$. We call such a basis an \emph{exceptional configuration\/},
which as before is an orthogonal basis for $\hbox{Cl}(X)$ with respect to the
intersection form.

\subsection*{Ordered and unordered Configuration types}
We saw above that lines through two or more points
give rise to classes of prime divisors of negative self-intersection.
Similarly, if instead the points $p_{i_1},\dots,p_{i_r}$ lie
on an irreducible conic and none of the other points lie
on that same conic, then the class of the proper transform of that
conic is $D=2L-E_{i_1}-\cdots-E_{i_r}$ and $D$ is the class of a prime divisor
of self-intersection $D^2=4-r$, hence negative if $r>4$.
Instead of just lines through 2 or more points, in
the context of algebraic geometry what is of interest
is more generally the set of all prime divisors of negative self-intersection.

We can formalize this generalized notion as a \emph{configuration type\/}
of points. Up to equivalence, an \emph{ordered\/} configuration type of $n$ points
in the plane is a matrix $T$ with $n+1$ columns whose rows
satisfy two conditions: negative self-intersection and
pairwise nonnegativity (explained below).
Two matrices satisfying these two conditions will be regarded as giving the same
ordered configuration type if one matrix can be obtained from the
other by permuting its rows. (We will say that
two matrices satisfying the two conditions will be regarded as giving the same
\emph{unordered\/} configuration type if one matrix can be obtained from the
other by permuting either its rows or columns or both.)

The two conditions come from our wanting the rows to specify the
coefficients, with respect to the basis $L,E_1,\ldots,E_n$, of
classes of prime divisors of negative self-intersection. (Although
in general there can be infinitely many classes of prime divisors
of negative self-intersection, if $n\le8$, it is known there are
only finitely many. See Lemma \ref{NEGisfinite} for the case of interest
here, $n=6$; the case for any $n\le 8$ is similar.)

Thus if $(d,m_1,\ldots,m_n)$ is a row of the matrix $T$,
we require $d^2-m_1^2-\cdots-m_n^2<0$ (negative self-intersection), and if
$(d',m'_1,\ldots,m'_n)$ is another row of the matrix,
we require $dd'-m_1m'_1-\cdots-m_nm'_n\ge0$
(pairwise nonnegativity), corresponding to
an intersection theoretic version of
Bezout's theorem, saying that $C\cdot D\ge0$
if $C$ and $D$ are prime divisors with $C\ne D$.
We will say a configuration type $T$ is
\emph{representable\/} if there is a set of essentially distinct
points $p_1,\ldots,p_n$ giving a surface $X$ such that the rows of $T$
are (in terms of the exceptional configuration $L,E_1,\ldots,E_n$ for $X$)
the classes of all prime divisors of negative self-intersection on $X$.

\subsection*{Goals and Motivation}
The goal of this paper is to classify all of the configuration types
for $n=6$ essentially distinct points of $\pr^2$ which when blown up
give a surface $X$ for which $-K_X$ is nef, and to determine
representability for each configuration type.
In order to formally write down possible matrices, we must
have a set $S$ of possible vectors $(d,m_1,\ldots,m_n)$
to draw from. In principle, $S$ should consist of all coefficient
vectors which occur for prime divisors of negative self-intersection for any
prime divisor that occurs for any choice of the points $p_i$.
Then we can attempt to write down all
possible matrices satisfying the two given conditions
(of negative self-intersection and pairwise nonnegativity)
where each row is chosen from $S$. Having written down
all possible matrices, we can consider representability:
i.e., for which matrices is there an algebraically closed
field $k$ and an actual set of points $p_i$ in $\pr^2_k$
such that the set of prime divisors on $X$ is
exactly that specified by the matrix.

The underlying motivation for carrying out this classification is that,
if $n\le8$, then two sets of points, $p_1,\ldots,p_n$ and $p'_1,\ldots,p'_n$,
have the same ordered configuration type if and only if, for all choices of nonnegative
integers $m_1,\ldots,m_n$, the Hilbert functions of the fat point subschemes
$m_1p_1+\cdots+m_np_n$ and $m_1p'_1+\cdots+m_np'_n$ are the same
\cite{refGH}. (We recall the notions of fat points, their ideals and their Hilbert functions
in Section \ref{bkgd}, and their
graded Betti numbers in Section \ref{resols}.)

\subsection*{Previous Work}
We classified the configuration types for sets of $n=6$ distinct points of $\pr^2$ in \cite{refGH},
and we also showed that if $p_1,\ldots,p_n$ and $p'_1,\ldots,p'_n$
have the same ordered configuration type, then for any nonnegative integers
$m_1,\ldots,m_n$, the graded Betti numbers of the ideals
$I(m_1p_1+\cdots+m_np_n)$ and $I(m_1p'_1+\cdots+m_np'_n)$
defining the fat point subschemes are the same.
(In \cite{refGH} for efficiency we listed only the unordered configuration types,
of which there are 11. These 11 comprise 353 ordered configuration types, but
two ordered types with the same unordered type differ only in the indexation of
the points. For example, one of the 11 is the situation where 3 points in a set of 6 points
is collinear, and otherwise no more than 2 of the 6 points is collinear. There are $\binom{6}{3}=20$
essentially different ways to number the 6 points, so this one unordered type comprises
20 ordered types.
Thus there is little reason to explicitly list the ordered types,
and we will normally only explicitly list unordered types,
as was done in \cite{refGH}.)

Since the graded Betti numbers determine the Hilbert function, and
since knowing the Hilbert functions of $m_1p_1+\cdots+m_6p_6$ for
all choices of the $m_i$ allows one to determine the set of prime
divisors on $X$ of negative self-intersection and hence to recover
the configuration type of the points, this shows that a
classification of configuration types of 6 distinct points in the
plane is actually a classification of the points up to graded
Betti numbers (i.e., where we regard two sets of 6 points
$p_1,\ldots,p_6$ and $p'_1,\ldots,p'_6$ as equivalent if the
graded Betti numbers of $I(m_1p_1+\cdots+m_6p_6)$ and
$I(m_1p'_1+\cdots+m_6p'_6)$ are the same for all choices of
nonnegative integers $m_i$).

\subsection*{Results}
In this paper we consider the classification of 6 essentially
distinct points, but for both technical and practical reasons we
do so only under the restriction that the anticanonical divisor
$-K_X$ on $X$ is nef. With this restriction, we show that every
type is representable over every algebraically closed field $k$
and we show that a classification by type is equivalent to a
classification up to graded Betti numbers. We also give an
explicit procedure for determining the graded Betti numbers for
the ideal $I(Z)$ for any fat point subscheme
$Z=m_1p_1+\cdots+m_6p_6$ supported at the six points, given only
the coefficients $m_i$ and the ordered configuration type of the
points. While this procedure can easily be carried out by hand, an
awk script automating the procedure can be run over the web at
\url{http://www.math.unl.edu/~bharbourne1/6ptsNef-K/6reswebsite.html}.
For some examples, see Section \ref{exmpls}.

The problem of determining all possible Hilbert functions and
graded Betti numbers for arbitrary fat point subschemes
$2p_1+\cdots+2p_n$, and of determining the configurations of the
points that give rise to the different possibilities, was raised
in \cite{refGMS}. Thus \cite{refGH} completely solves the problem
for $n=6$ in the original context of distinct points, not only for
double points but for fat point schemes $m_1p_1+\cdots+m_6p_6$
with $m_i$ arbitrary. What we do here likewise completely solves
the problem for arbitrary $m_i$, in the case of 6 essentially
distinct points under the condition that $-K_X$ is nef. Indeed,
what we find is that there are 90 different unordered
configuration types, corresponding to equivalence classes of
matrices whose rows are drawn from a certain set $S$ as discussed
above and given explicitly in Lemma \ref{NEGisfinite}. (If we were to
remove the restriction that $-K_X$ is nef, we would, in addition
to what is specified in Lemma \ref{NEGisfinite}, also have to include in
$S$ the coefficient vectors of all classes of the form
$E_i-E_{j_1}-\cdots-E_{j_r}$ for all subsets
$\{j_1,\ldots,j_r\}\subsetneq\{1,\ldots,6\}$ with $r>1$ and $i<j_l$
for all $l$, and also all classes of the form
$L-E_{i_1}-\cdots-E_{i_l}$ for all $0<i_1<\ldots<i_l\le6$ with
$l>3$. This results in many more configuration types. Having
$-K_X$ be nef also affords technical simplifications in computing
generators for dual cones given generators for a cone, which we
need to do for our method of proving that the graded Betti numbers
of $I(m_1p_1+\cdots+m_6p_6)$ depend only on the coefficients $m_i$
and the configuration type of the points $p_i$.)

The condition that $-K_X$ be nef is fairly reasonable,
both algebraically and geometrically. Algebraically,
one of the cases of most interest is the uniform case,
i.e., cases where the fat point subscheme $Z$ is of the form
$Z=mp_1+\cdots+mp_6$. Also, one typically considers
schemes $Z$ only which satisfy the proximity inequalities
(see Section \ref{resols}), and if $-K_X$ is nef, then a uniform $Z$
satisfies the proximity inequalities if and only if $m\ge0$.

Geometrically, the surfaces obtained by blowing up 6 essentially
distinct points of $\pr^2$ such that $-K_X$ is nef are precisely
the surfaces which occur by resolving the singularities of normal
cubic surfaces in $\pr^3$. Thus this paper can be regarded as a
contribution to the long history of work on cubic surfaces. The
classification of normal cubic surfaces up to the types of their
singularities (as given by the Dynkin diagrams of the singular
points) is classical, at least in characteristic 0 (see
\cite{refBW}). What is new here is first the (relatively easy)
classification of the corresponding configuration types of points
in $\pr^2$. (Resolving the singularities of a normal cubic surface
gives a surface $X$ for which $-K_X$ is nef, but each $X$
typically has several birational morphisms to $\pr^2$, and each
such morphism gives a set of 6 points in $\pr^2$ which when blown
up give $X$. Thus typically several configuration types occur for
each Dynkin diagram.) It is much harder to show that the
configuration type of the points $p_i$ is enough together with the
coefficients $m_i$ to determine the graded Betti numbers of
$I(m_1p_1+\cdots+m_6p_6)$. When the points are distinct we showed
this in \cite{refGH} without requiring $-K_X$ be nef. What is new
here is that we show this for points that can be infinitely near,
but under the assumption that $-K_X$ is nef. (It was already
known, as a consequence of Theorem 8 of \cite{refBHProc}, that the
configuration type of the points $p_i$ is enough, together with
the coefficients $m_i$, to determine the Hilbert function of
$I(m_1p_1+\cdots+m_np_n)$ for any $n\le 8$ essentially distinct
points of $\pr^2$, whether $-K_X$ is nef or not.)

\section{Background}\label{bkgd}

We recall here some of the background we will need
on fat points and on surfaces obtained by blowing up
essentially distinct points of $\pr^2$. We work over an
algebraically closed field $k$ of arbitrary characteristic.

A fat point subscheme $Z=m_1p_1+\cdots+m_np_n$ usually is considered
in the case that the points $\{p_i\}$ are distinct points.
In particular, let $p_1,\ldots,p_n$ be distinct points
of $\pr^2$. Given nonnegative integers $m_i$, the fat point
subscheme $Z=m_1p_1+\cdots+m_np_n\subset \pr^2$ is
defined to be the subscheme defined by the ideal
$I(Z)=I(p_1)^{m_1}\cap\cdots\cap I(p_n)^{m_n}$, where
$I(p_i)\subseteq R=k[\pr^2]$ is the ideal generated by all forms
(in the polynomial ring $R$ in three variables over the field $k$)
vanishing at $p_i$. The \emph{support\/} of $Z$ consists of
the points $p_i$ for which $m_i$ is positive.
For another perspective, let $\C I_Z$
be the sheaf of ideals defining $Z$ as a subscheme of $\pr^2$.
Now let $X$ be obtained by blowing up the points $p_i$.
Given a divisor $F$ we will denote the
corresponding line bundle by $\C O_X(F)$. With this convention,
$\C I_Z=\pi_*(\C O_X(-m_1E_1-\cdots-m_nE_n))$
and the stalks of $\C I_Z$ are complete ideals (as
defined in \cite{refZ} and \cite{refZS})
in the local rings of the structure sheaf of $\pr^2$.
We can recover $I(Z)$ from $\C I_Z$ since
the homogeneous component $I(Z)_t$ of $I(Z)$ of degree $t$
is just $H^0(X, \C I_Z(t))$.

We can just as well consider essentially distinct points
$p_1,\ldots,p_n\in\pr^2$. Again let $\pi:X\to \pr^2$ be given by
blowing up the points $p_i$, in order. We define the fat point
subscheme $Z=m_1p_1+\cdots+m_np_n$ to be the subscheme whose ideal
sheaf is the coherent sheaf of ideals $\pi_*(\C
O_X(-m_1E_1-\cdots-m_nE_n))$. Note that the stalks of $\pi_*(\C
O_X(-m_1E_1-\cdots-m_nE_n))$ are again complete ideals in the
stalks of the local rings of the structure sheaf of $\pr^2$, and,
conversely, if \C I is a coherent sheaf of ideals on $\pr^2$ whose
stalks are complete ideals and if \C I defines a 0-dimensional
subscheme, then there are essentially distinct points
$p_1,\ldots,p_n$ of $\pr^2$ and integers $m_i$ such that with
respect to the corresponding exceptional configuration we have $\C
I=\pi_*(\C O_X(-m_1E_1-\cdots-m_nE_n))$ (see \cite{refAppendix},
\cite{refZ} and \cite{refZS} for more details). As before we
define $I(Z)$ to be the ideal in $R$ given as $I(Z)=\oplus_{t\ge
0} H^0(X, \C I_Z(t))$. The Hilbert function of a homogenous ideal
$I\subseteq R$ is just the function $h_I(t)=\hbox{dim } I_t$ giving
the vector space dimension of the homogeneous component $I_t$ of
$I$ as a function of the degree $t$. The Hilbert function of a fat
point subscheme $Z$ will be the function $h_Z(t)=\hbox{dim}
(R/I)_t$ giving the vector space dimension of the homogeneous
components of the quotient ring $R/I$ as a function of degree.
Note that $h_{I(Z)}(t)+h_Z(t)=\binom{t+2}{2}$. (We recall in
Section \ref{resols} the notions of the minimal free resolution of $I(Z)$
and its graded Betti numbers.)

Every smooth projective surface $X$ with a birational morphism
to $\pr^2$ arises as a blow up of $n$ essentially
distinct points, where $n$ is uniquely determined
by $X$, since $n+1$ is the rank of $\hbox{Cl}(X)$ as a free abelian
group. Since here we are interested in the case $n=6$,
we will always hereafter assume that $n=6$.
We will also mainly be interested in those $X$ for which the anticanonical
class is nef. The anticanonical class has an intrinsic definition,
but in terms of an exceptional configuration it is
always $3L-E_1-\cdots-E_n$. A divisor (or divisor class) $F$ being \emph{nef\/}
means that $F\cdot D\ge 0$
whenever $D$ is the class of an effective divisor
(with \emph{effective\/} meaning that $D$ is a nonnegative integer linear combination
of reduced irreducible curves).

We now recall the connection of normal cubic surfaces
with blow ups $X$ of $\pr^2$ at 6 essentially distinct
points such that $-K_X$ is nef. If $-K_X$ is nef, by
Lemma \ref{NEGisfinite} the linear system
$|-K_X|$ has no base points so it
defines a morphism $\phi_{|-K_X|}:X\to\pr^3$,
whose image is a cubic surface.
By Proposition 3.2 of \cite{refBirMor}, the image of $\phi_{|-K_X|}$
is normal, obtained by contracting
to a point every prime divisor orthogonal to $-K_X$
(i.e., every smooth rational curve of self-intersection $-2$).
In fact, the images of the $(-2)$-curves are rational double points,
and the inverse image of each singular point is a minimal resolution
of the singularity.
It is not hard to check that the subgroup $K_X^\perp\subsetneq\hbox{Cl}(X)$
of all divisor classes orthogonal to $K_X$ is negative definite.
Thus Theorem 2.7 and Figure 2.8, both of \cite{refA}, apply;
i.e., the intersection graph of a fiber over a singular point
is a Dynkin diagram of type $A_i$, $D_i$ or $E_i$.
The combinations of Dynkin diagrams that occur
for the singularities on a single surface are well known.
A determination in characteristic 0 is given in \cite{refBW}.
We recover that result for any characteristic; see
Table \ref{configtable}.

To state the next result, let $\hbox{NEG}(X)$
denote the set of classes of prime divisors of
negative self-intersection on a surface $X$
obtained by blowing up 6 essentially distinct points of $\pr^2$.
Let $\C B=\{E_i: i>0\}$
($\C B$ here is for \emph{blow up} of a point),
$\C V=\{E_i-E_{i_1}-\cdots-E_{i_r}: r\ge 1, 0<i<i_1<\cdots<i_r\le 6\}$
($\C V$ here is for \emph{vertical}),
$\C L=\{L-E_{i_1}-\cdots-E_{i_r}: r\ge 2, 0<i_1<\cdots<i_r\le 6\}$
($\C L$ here is for points on a \emph{line}), and
$\C Q=\{2L-E_{i_1}-\cdots-E_{i_r}: r\ge 5, 0<i_1<\cdots<i_r\le 6\}$
($\C Q$ here is for points on a conic, defined by a \emph{quadratic} equation).
Also, let $\C B'=\C B$, $\C V'=\{E_i-E_j:  0<i<j\le 6\}$,
$\C L'=\{L-E_i-E_j: 0<i<j\le 6\}\cup\{L-E_i-E_j-E_k: 0<i<j<k\le 6\}$, and
$\C Q'=\C Q$, and
let $\C V''=\C V'$,
$\C L''=\{L-E_i-E_j-E_k: 0<i<j<k\le 6\}$, and
$\C Q''=\{2L-E_1-\cdots-E_6\}$.

\begin{lem}\label{NEGisfinite}
Let $X$ be obtained by blowing up
6 essentially distinct points of $\pr^2$. Then the following hold:
\begin{itemize}
\item[(a)] $\hbox{NEG}(X)\subseteq \C B\cup\C V\cup\C L\cup\C Q$,
and every class in $\hbox{NEG}(X)$ is the class
of a smooth rational curve;
\item[(b)] if moreover $-K_X$ is nef, then
$\hbox{NEG}(X)\subseteq \C B'\cup\C V'\cup\C L'\cup\C Q'$;
\item[(c)] for any nef $F\in\hbox{Cl}(X)$, $F$ is effective
(hence $h^2(X, F)=0$ by duality), $|F|$ is base point free,
$h^0(X, F) = (F^2-K_X\cdot F)/2 + 1$ and $h^1(X, F)=0$;
\item[(d)] $\hbox{NEG}(X)$ generates the subsemigroup
$\hbox{EFF}(X)\subsetneq\hbox{Cl}(X)$ of classes of effective divisors; and
\item[(e)] any class $F$ is nef if and only if
$F\cdot C\ge0$ for all $C\in \hbox{NEG}(X)$.
\end{itemize}
\end{lem}

\begin{proof} This result is well known. A proof of parts (a), (c), (d) and
(e) when the points are assumed to be distinct is given in detail
in \cite{refGH}. The same proof carries over with only minor
changes here. Part (b) follows from (a) just by taking into
account that each class $C$ in $\hbox{NEG}(X)$ must satisfy
$-K_X\cdot C\ge0$.
\end{proof}

\begin{rem}\label{minustwoRem} In the same way that it is easier
to specify a combinatorial geometry of points in the plane by
specifying which sets of three or more points are collinear
(suppressing mention of all of the pairs of points
which define a line going through no other point),
it is often easier to work with the set
$\hbox{neg}(X)=\{C\in \hbox{NEG}(X) : C^2 < -1\}$
than with $\hbox{NEG}(X)$. As shown in Remark 2.2 of \cite{refGH},
$\hbox{neg}(X)$ determines $\hbox{NEG}(X)$.
In fact, we have:
$$\hbox{NEG}(X)=\hbox{neg}(X)\cup
\{C\in\C B\cup\C L\cup\C Q\hbox{ $|$ }C^2=-1,
C\cdot D\ge0\ \hbox{ for all }\ D\in \hbox{neg}(X)\}.$$
If $-K_X$ is nef, note that
$\hbox{neg}(X)\subseteq \C V''\cup\C L''\cup\C Q''$.
\end{rem}

\section{Configuration Types}\label{conftypes}

In this section we determine the configuration types
of 6 essentially distinct points of $\pr^2$, under the restriction that
$-K_X$ is nef. I.e., we find all \emph{pairwise nonnegative\/} subsets of
$\C B'\cup\C V'\cup\C L'\cup\C Q'$ (a pairwise nonnegative
subset being a subset such that
whenever $C$ and $D$ are distinct elements of the subset,
we have $C\cdot D\ge 0$). With Remark \ref{minustwoRem} in mind,
we actually only do this for subsets of
$\C V''\cup\C L''\cup\C Q''$.
Also, we do this only up to permutations
of the classes $E_1,\ldots,E_6$. Thus we find the unordered configuration
types, hence only one representative
of each orbit under the action of the group of permutations
of $E_1,\ldots,E_6$. (Note for example that
$\{E_1-E_3, E_2-E_4\}$ and $\{E_1-E_2, E_3-E_4\}$ are the same
up to permutations of the $E_i$.)

We also show that each
configuration type actually occurs over every algebraically closed
field (regardless of the characteristic).
Both for this latter question of representability and for
distinguishing when different pairwise nonnegative
subsets $T$ give different configuration types, it is helpful
to compute the torsion groups $\hbox{Tors}_T$ for the quotients
$\hbox{Cl}(X)/\langle T\rangle$ of the divisor class group
by the subgroup generated by the elements of $T$ (or equivalently, the torsion
subgroup of $K_X^\perp/\langle T\rangle$). So we include this
information in Table \ref{configtable}, whenever $\hbox{Tors}_T\ne0$.

We can associate a graph (whose connected components are Dynkin
diagrams \cite{refHandbook}) to each configuration type. If $T$ is
a pairwise nonnegative subset of $\C V''\cup\C L''\cup\C Q''$, we
have the graph $G_T$, whose vertices are the elements of $T$ and
we have $C\cdot D$ edges between each distinct pair of vertices
$C,D\in T$. It turns out that there is at most one edge between
any two vertices, and, in terms of the standard notation for
Dynkin diagrams, the connected components of each $G_T$ are always
among the following types: $A_i$, $1\le i\le5$; $D_4$; $D_5$ and
$E_6$. (If the graph $G_T$ for a subset $T$ has more than one
connected component, say an $A_1$ and two of type $A_2$, we write
this as $A_12A_2$.) Different configuration types can have the
same graph, but the torsion subgroup for each configuration type
turns out to be determined by the graph. Since different
configuration types can have the same graph, the Dynkin diagram
(such as $A_12A_2$) is not by itself enough to uniquely identify a
configuration type, so we distinguish different configuration
types with the same graph by appending a lower case letter (for
example, $A_12A_2a$ or $A_12A_2b$) when there is more than one
configuration type with a given graph.

The 90 different configuration types (i.e., the classification,
up to permutations, of the
pairwise nonnegative subsets of $\C V''\cup\C L''\cup\C Q''$)
are shown in Table \ref{configtable}. For each configuration
type $T$ we give the corresponding graph $G_T$ (using Dynkin notation),
we give the set $T$ itself, and, when not 0, we give $\hbox{Tors}_T$ (which is always either
0, ${\bf Z}/2{\bf Z}$ or ${\bf Z}/3{\bf Z}$; we denote the latter two
by ${\bf Z}_2$ and ${\bf Z}_3$ in Table \ref{configtable}).
We give the set $T$ by listing its elements, following the
approach used in \cite{refBCH}. We use letters A through F
to denote the points $p_1$ through $p_6$, and numbers to indicate
the degree of the curve.
For example, the set $T$ for $3A_1d$ is given
as \hbox{0: AB, CD; 2: ABCDEF}. Thus $T$ consists of
the classes $E_1-E_2$, $E_3-E_4$ and $2L-E_1-\cdots-E_6$.

We obtained the table by brute force as follows. Start by finding all
single element configuration types, which is easy.
These are just the single element subsets of
$\C V''\cup\C L''\cup\C Q''$. Pick a representative
for each orbit under the permutation action.
We get three singleton sets $T$,
corresponding to items 2, 3 and 4 in Table \ref{configtable}.
Add to each singleton configuration type $T$ each element of
$\C V''\cup\C L''\cup\C Q''$ which meets every
class already in $T$ nonnegatively, and again
pick a representative set from each orbit.
Continue this way for six cycles. (Six is enough since, as shown in the proof of Proposition \ref{rprsntblty},
the elements of each $T$ are linearly independent, and so
$T$ can have at most 6 elements.)

\begin{prop}\label{rprsntblty}
Over every algebraically closed field $k$,
each configuration type occurs
as $\hbox{neg}(X)$
for some surface $X$ obtained by blowing up
6 essentially distinct points of $\pr^2_k$.
\end{prop}

\begin{proof} Let $T$ be the set of classes
of a configuration type, and consider the group
$K^\perp/\langle T\rangle$. Since $\C V''\cup\C L''\cup\C Q''$
is a finite set and since from Table \ref{configtable} we see that the torsion subgroup
of $K^\perp/\langle T\rangle$ is either trivial or has prime order,
we can pick a squarefree positive integer $l$ and a surjective homomorphism
$\phi: K^\perp/\langle T\rangle\to {\bf Z}/l{\bf Z}$ such that
no element $C\in\C V''\cup\C L''\cup\C Q''$ not already in
$\langle T\rangle$ maps to 0 in ${\bf Z}/l{\bf Z}$.

Now let $C$ be a non-supersingular
smooth plane cubic curve.
Since $C$ is not supersingular, $\hbox{Pic}^0(C)$
has a subgroup isomorphic to ${\bf Z}/l{\bf Z}$; we identify
${\bf Z}/l{\bf Z}$ with this subgroup of $\hbox{Pic}^0(C)$. Thus there is
a homomorphism $\Phi: \hbox{Cl}(X)\to \hbox{Pic}(C)$
such that the image $\Phi(K_X^\perp)$ is exactly ${\bf Z}/l{\bf Z}$.
Pick any point on $C$ to be $p_1$. Then pick $p_i$ to be the image
of $E_i-E_1$ in $\hbox{Pic}^0(C)$.

Under the usual identification
of $\hbox{Pic}^0(C)$ with $C$ itself, this gives us
six points $p_1,\ldots, p_6$. (It may be that
some of the points are formally the same. For example,
if $E_1-E_2\in T$, then $p_1=p_2$. This just means
that $p_2$ is the point on the proper transform
$C'$ of $C$ on $X_1$ infinitely near
$p_1\in X_0$. Since restricting the mapping $\pi_1:X_1\to\pr^2$ to $C'$
gives an isomorphism of $C'$ to $C$, there is a natural identification of
$C'$ with $C$. Under this identification we can indeed
regard $p_1$ and $p_2$ as being the same point of $C$,
even though properly speaking $p_1\in\pr^2$ and $p_2\in X_1$.)

By construction, the surface $X$ obtained by
blowing up the points $p_1,\ldots, p_6$
has the property that an element $D\in\C V''\cup\C L''\cup\C Q''$
is in the kernel of $\Phi$ if and only if $D\in \langle T\rangle$.
By \cite{refTAMS}, an element $D$ of $\C V''\cup\C L''\cup\C Q''$
is the class of an effective divisor if and only if
$D\in\hbox{ker}(\Phi)$. Thus $D\in\C V''\cup\C L''\cup\C Q''$
is effective if and only if $D\in \langle T\rangle$.

In particular, the elements
of $T$ are effective and $\hbox{neg}(X)\subseteq \langle T\rangle$.
Let $D\in \hbox{neg}(X)$. By Lemma \ref{NEGisfinite}(b), $D^2=-2$.
Now write $D$ as an integer linear
combination of elements of $T$. Thus we can write
$D=D_1-D_2$, where $D_1$ is a sum of elements of $T$
with positive coefficients and $D_2$ is either 0 or
a sum of different elements of $T$
with positive coefficients. Note that $D_1$ is not zero, since
otherwise $D$ is either 0 or antieffective, neither of which can hold
since $D$ is the class of a prime divisor.
We claim however that $D_2=0$.
If not, then, since $K_X^\perp$ is negative definite and even,
we have $D_1^2\le -2$ and $D_2^2\le -2$. Since $D_1$ and $D_2$ involve
different elements of $T$ (which therefore meet nonnegatively),
we also see $D_1\cdot D_2\ge 0$. Thus $D^2=D_1^2-2D_1\cdot D_2+D_2^2\le-4$,
contradicting $D^2=-2$. (A similar argument shows that
the elements of $T$ are linearly independent.
If not, we can find an expression $D_1-D_2=0$ for some 
nonnegative linear integer combinations $D_i$  of elements of 
disjoint subsets $T_i\subseteq T$. 
By pairwise nonnegativity, we have $D_1\cdot D_2\ge 0$,
but $-K_X^\perp$ is negative definite, so $0\ge D_i^2=D_1\cdot D_2$,
hence $D_i^2=0$, so $D_i=0$. But $T\subseteq\C V''\cup\C L''\cup\C Q''$,
and every element of $\C V''\cup\C L''\cup\C Q''$ meets
$A=14L-6E_1-5E_2-4E_3-3E_4-2E_5-E_6$ positively, so 
if $D_i$ is not a linear integer combination of elements of $T_i$ 
with 0 coefficients, then we have $0<A\cdot D_i=A\cdot 0=0$,
which is impossible. Thus each $D_i$ is the trivial linear combination,
hence $T$ is linearly independent.)

Thus every element of $\hbox{neg}(X)$ is a nonnegative
sum of elements of $T$, each of which
is effective. But the elements of $\hbox{neg}(X)$ are prime divisors of negative
self-intersection, hence each can be written as a nonnegative sum of classes of effective
divisors only one way; i.e., every element of $\hbox{neg}(X)$ is an element of $T$.

By Lemma \ref{NEGisfinite}(d), every element of $T$ is a nonnegative integer linear
combination of elements of $\hbox{NEG}(X)$. But $T\subseteq K_X^\perp$,
so in fact every element of $T$ is a nonnegative integer linear
combination of elements of $\hbox{neg}(X)$. Since
$\hbox{neg}(X)\subseteq T$, and since $T$ is linearly independent,
this is possible only if $\hbox{neg}(X)=T$.
\end{proof}

%\vfil\eject

\ \vskip.1in

{
\newcount\papercnt
\papercnt=1
\def\pprno{\scriptsize\number\papercnt.$\,\,$\global\advance\papercnt by1}

\begin{table}
\hbox to\hsize{\vbox{
\hbox to 2.7in{\hrulefill}
\vskip-.05in
\hbox to 2.7in{\scriptsize\hbox to 0.9in{\ \ \ \ $G_T$ \hfil} $T$ \hfil $\hbox{Tors}_T$}
\vskip-.1in
\hbox to 2.7in{\hrulefill}

\hbox to 2.7in{\scriptsize\hbox to 0.9in{\ $\,$\pprno \ \ $\emptyset$ \hfil}\hfil}

\hbox to 2.7in{\scriptsize\hbox to 0.9in{\ $\,$\pprno $A_1a$ \hfil} 0: AB \hfil}

\hbox to 2.7in{\scriptsize\hbox to 0.9in{\ $\,$\pprno $A_1b$ \hfil} 1: ABC \hfil}

\hbox to 2.7in{\scriptsize\hbox to 0.9in{\ $\,$\pprno $A_1c$ \hfil} 2: ABCDEF \hfil}

\hbox to 2.7in{\scriptsize\hbox to 0.9in{\ $\,$\pprno $2A_1a$ \hfil} 0: AB, CD \hfil}

\hbox to 2.7in{\scriptsize\hbox to 0.9in{\ $\,$\pprno $2A_1b$ \hfil} 0: AB; 1: ABC \hfil}

\hbox to 2.7in{\scriptsize\hbox to 0.9in{\ $\,$\pprno $2A_1c$ \hfil} 0: DE; 1: ABC \hfil}

\hbox to 2.7in{\scriptsize\hbox to 0.9in{\ $\,$\pprno $2A_1d$ \hfil} 0: AB; 2: ABCDEF \hfil}

\hbox to 2.7in{\scriptsize\hbox to 0.9in{\ $\,$\pprno $2A_1e$ \hfil} 1: ABC, ADE \hfil}

\hbox to 2.7in{\scriptsize\hbox to 0.9in{\pprno $A_2a$ \hfil} 0: AB, BC \hfil}

\hbox to 2.7in{\scriptsize\hbox to 0.9in{\pprno $A_2b$ \hfil} 0: CD; 1: ABC \hfil}

\hbox to 2.7in{\scriptsize\hbox to 0.9in{\pprno $A_2c$ \hfil} 1: ABC, DEF \hfil}

\hbox to 2.7in{\scriptsize\hbox to 0.9in{\pprno $3A_1a$ \hfil} 0: AB, CD, EF \hfil}

\hbox to 2.7in{\scriptsize\hbox to 0.9in{\pprno $3A_1b$ \hfil} 0: AB, DE; 1: ABC \hfil}

\hbox to 2.7in{\scriptsize\hbox to 0.9in{\pprno $3A_1c$ \hfil} 0: BC; 1: ABC, ADE \hfil}

\hbox to 2.7in{\scriptsize\hbox to 0.9in{\pprno $3A_1d$ \hfil} 0: AB, CD; 2: ABCDEF \hfil}

\hbox to 2.7in{\scriptsize\hbox to 0.9in{\pprno $3A_1e$ \hfil} 1: ABC, ADE, BDF \hfil}

\hbox to 2.7in{\scriptsize\hbox to 0.9in{\pprno $A_1A_2a$ \hfil} 0: AB, BC, DE \hfil}

\hbox to 2.7in{\scriptsize\hbox to 0.9in{\pprno $A_1A_2b$ \hfil} 0: AB, BC; 1: ABC \hfil}

\hbox to 2.7in{\scriptsize\hbox to 0.9in{\pprno $A_1A_2c$ \hfil} 0: AB, BC; 1: DEF \hfil}

\hbox to 2.7in{\scriptsize\hbox to 0.9in{\pprno $A_1A_2d$ \hfil} 0: AB, CD; 1: ABC \hfil}

\hbox to 2.7in{\scriptsize\hbox to 0.9in{\pprno $A_1A_2e$ \hfil} 0: CD, EF; 1: ABC \hfil}

\hbox to 2.7in{\scriptsize\hbox to 0.9in{\pprno $A_1A_2f$ \hfil} 0: CD; 1: ABC, AEF \hfil}

\hbox to 2.7in{\scriptsize\hbox to 0.9in{\pprno $A_1A_2g$ \hfil} 0: AB; 1: ABC, ADE \hfil}

\hbox to 2.7in{\scriptsize\hbox to 0.9in{\pprno $A_1A_2h$ \hfil} 0: AB; 1: ABC, DEF \hfil}

\hbox to 2.7in{\scriptsize\hbox to 0.9in{\pprno $A_1A_2i$ \hfil} 0: AB, BC; 2: ABCDEF \hfil}

\hbox to 2.7in{\scriptsize\hbox to 0.9in{\pprno $A_3a$ \hfil} 0: AB, BC, CD \hfil}

\hbox to 2.7in{\scriptsize\hbox to 0.9in{\pprno $A_3b$ \hfil} 0: CD, DE; 1: ABC \hfil}

\hbox to 2.7in{\scriptsize\hbox to 0.9in{\pprno $A_3c$ \hfil} 0: AB, DE; 1: ACD \hfil}

\hbox to 2.7in{\scriptsize\hbox to 0.9in{\pprno $A_3d$ \hfil} 0: AF; 1: ABC, ADE \hfil}

\hbox to 2.7in{\scriptsize\hbox to 0.9in{\pprno $A_3e$ \hfil} 0: BC, CD; 1: ABC \hfil}

\hbox to 2.7in{\scriptsize\hbox to 0.9in{\pprno $4A_1a$ \hfil} 0: BC, DE; 1: ABC, ADE \hfil ${\bf Z}_2$}

\hbox to 2.7in{\scriptsize\hbox to 0.9in{\pprno $4A_1b$ \hfil} 0: AB, CD, EF; 2: ABCDEF \hfil ${\bf Z}_2$}

\hbox to 2.7in{\scriptsize\hbox to 0.9in{\pprno $4A_1c$ \hfil} 1: ABC, ADE, BDF, CEF \hfil ${\bf Z}_2$}

\hbox to 2.7in{\scriptsize\hbox to 0.9in{\pprno $2A_1A_2a$ \hfil} 0: AB, BC, DE; 1: ABC \hfil}

\hbox to 2.7in{\scriptsize\hbox to 0.9in{\pprno $2A_1A_2b$ \hfil} 0: AB, CD, EF; 1: ABC \hfil}

\hbox to 2.7in{\scriptsize\hbox to 0.9in{\pprno $2A_1A_2c$ \hfil} 0: AB, DE; 1: ABC, DEF \hfil}

\hbox to 2.7in{\scriptsize\hbox to 0.9in{\pprno $2A_1A_2d$ \hfil} 0: AB, DE, EF; 1: ABC \hfil}

\hbox to 2.7in{\scriptsize\hbox to 0.9in{\pprno $2A_1A_2e$ \hfil} 0: AB, DE; 1: ABC, ADE \hfil}

\hbox to 2.7in{\scriptsize\hbox to 0.9in{\pprno $2A_1A_2f$ \hfil} 0: BF, DE; 1: ABC, ADE \hfil}

\hbox to 2.7in{\scriptsize\hbox to 0.9in{\pprno $2A_1A_2g$ \hfil} 0: AC; 1: ABC, ADE, BDF \hfil}

\hbox to 2.7in{\scriptsize\hbox to 0.9in{\pprno $2A_1A_2h$ \hfil} 0: AB, BC, DE; 2: ABCDEF \hfil}

\hbox to 2.7in{\scriptsize\hbox to 0.9in{\pprno $A_1A_3a$ \hfil} 0: AB, BC, CD, EF \hfil}

\hbox to 2.7in{\scriptsize\hbox to 0.9in{\pprno $A_1A_3b$ \hfil} 0: AB, CD, DE; 1: ABC \hfil}

\hbox to 2.7in{\scriptsize\hbox to 0.9in{\pprno $A_1A_3c$ \hfil} 0: AB, DF; 1: ABC, ADE \hfil}
\vfil}
\hfill
\vbox{
\hbox to 3.2in{\hrulefill}
\vskip-.05in
\hbox to 3.2in{\scriptsize\hbox to 0.9in{\ \ \ \ $G_T$ \hfil} $T$ \hfil $\hbox{Tors}_T$}
\vskip-.1in
\hbox to 3.2in{\hrulefill}

\hbox to 3.2in{\scriptsize\hbox to 0.9in{\pprno $A_1A_3d$ \hfil} 0: AB, BC; 1: ABC, ADE \hfil}

\hbox to 3.2in{\scriptsize\hbox to 0.9in{\pprno $A_1A_3e$ \hfil} 0: AF, BC; 1: ABC, ADE \hfil}

\hbox to 3.2in{\scriptsize\hbox to 0.9in{\pprno $A_1A_3f$ \hfil} 0: BC, CD; 1: ABC, DEF \hfil}

\hbox to 3.2in{\scriptsize\hbox to 0.9in{\pprno $A_1A_3g$ \hfil} 0: BC, CD, EF; 1: ABC \hfil}

\hbox to 3.2in{\scriptsize\hbox to 0.9in{\pprno $A_1A_3h$ \hfil} 0: AB, BC, CD; 2: ABCDEF \hfil}

\hbox to 3.2in{\scriptsize\hbox to 0.9in{\pprno $2A_2a$ \hfil} 0: AB, BC, DE, EF \hfil}

\hbox to 3.2in{\scriptsize\hbox to 0.9in{\pprno $2A_2b$ \hfil} 0: AB, CF; 1: ABC, ADE \hfil}

\hbox to 3.2in{\scriptsize\hbox to 0.9in{\pprno $2A_2c$ \hfil} 0: AB, BC; 1: ABC, DEF \hfil}

\hbox to 3.2in{\scriptsize\hbox to 0.9in{\pprno $A_4a$ \hfil} 0: AB, BC, CD, DE \hfil}

\hbox to 3.2in{\scriptsize\hbox to 0.9in{\pprno $A_4b$ \hfil} 0: AB, CD, DE; 1: ACD \hfil}

\hbox to 3.2in{\scriptsize\hbox to 0.9in{\pprno $A_4c$ \hfil} 0: AB, BC, EF; 1: ADE \hfil}

\hbox to 3.2in{\scriptsize\hbox to 0.9in{\pprno $A_4d$ \hfil} 0: CD, DE, EF; 1: ABC \hfil}

\hbox to 3.2in{\scriptsize\hbox to 0.9in{\pprno $A_4e$ \hfil} 0: BC, CD; 1: ABC, BEF \hfil}

\hbox to 3.2in{\scriptsize\hbox to 0.9in{\pprno $A_4f$ \hfil} 0: AB, BC, CD; 1: ABC \hfil}

\hbox to 3.2in{\scriptsize\hbox to 0.9in{\pprno $D_4a$ \hfil} 0: BC, CD, DE; 1: ABC \hfil}

\hbox to 3.2in{\scriptsize\hbox to 0.9in{\pprno $D_4b$ \hfil} 0: AB, CD, EF; 1: ACE \hfil}

\hbox to 3.2in{\scriptsize\hbox to 0.9in{\pprno $A_12A_2a$ \hfil} 0: AB, BC, DE, EF; 1: ABC \hfil}

\hbox to 3.2in{\scriptsize\hbox to 0.9in{\pprno $A_12A_2b$ \hfil} 0: AB, CF, DE; 1: ABC, ADE \hfil}

\hbox to 3.2in{\scriptsize\hbox to 0.9in{\pprno $A_12A_2c$ \hfil} 0: AB, BC, DE; 1: ABC, DEF \hfil}

\hbox to 3.2in{\scriptsize\hbox to 0.9in{\pprno $A_12A_2d$ \hfil} 0: AB, CD; 1: ABC, AEF, CDE \hfil}

\hbox to 3.2in{\scriptsize\hbox to 0.9in{\pprno $A_12A_2e$ \hfil} 0: AB, BC, DE, EF; 2: ABCDEF \hfil}

\hbox to 3.2in{\scriptsize\hbox to 0.9in{\pprno $2A_1A_3a$ \hfil} 0: BC, CD, EF; 1: ABC, AEF \hfil ${\bf Z}_2$}

\hbox to 3.2in{\scriptsize\hbox to 0.9in{\pprno $2A_1A_3b$ \hfil} 0: AD, CE; 1: ABC, ADF, CEF \hfil ${\bf Z}_2$}

\hbox to 3.2in{\scriptsize\hbox to 0.9in{\pprno $2A_1A_3c$ \hfil} 0: AB, BC, DE; 1: ABC, ADE \hfil ${\bf Z}_2$}

\hbox to 3.2in{\scriptsize\hbox to 0.9in{\pprno $2A_1A_3d$ \hfil} 0: AF, BC, DE; 1: ABC, ADE \hfil ${\bf Z}_2$}

\hbox to 3.2in{\scriptsize\hbox to 0.9in{\pprno $2A_1A_3e$ \hfil} 0: BC, CF, DE; 1: ABC, ADE \hfil ${\bf Z}_2$}

\hbox to 3.2in{\scriptsize\hbox to 0.9in{\pprno $2A_1A_3f$ \hfil} 0: AB, BC, CD, EF; 2: ABCDEF \hfil ${\bf Z}_2$}

\hbox to 3.2in{\scriptsize\hbox to 0.9in{\pprno $A_1A_4a$ \hfil} 0: AB, BC, CD, EF; 1: ABC \hfil}

\hbox to 3.2in{\scriptsize\hbox to 0.9in{\pprno $A_1A_4b$ \hfil} 0: AB, CD, DE, EF; 1: ABC \hfil}

\hbox to 3.2in{\scriptsize\hbox to 0.9in{\pprno $A_1A_4c$ \hfil} 0: AB, DE, EF; 1: ABC, ADE \hfil}

\hbox to 3.2in{\scriptsize\hbox to 0.9in{\pprno $A_1A_4d$ \hfil} 0: AB, BF, DE; 1: ABC, ADE \hfil}

\hbox to 3.2in{\scriptsize\hbox to 0.9in{\pprno $A_1A_4e$ \hfil} 0: AB, BC, EF; 1: ABC, ADE \hfil}

\hbox to 3.2in{\scriptsize\hbox to 0.9in{\pprno $A_1A_4f$ \hfil} 0: AB, BC, CD, DE; 2: ABCDEF \hfil}

\hbox to 3.2in{\scriptsize\hbox to 0.9in{\pprno $A_5a$ \hfil} 0: AB, BC, CD, DE, EF \hfil}

\hbox to 3.2in{\scriptsize\hbox to 0.9in{\pprno $A_5b$ \hfil} 0: AB, BC, DE, EF; 1: ADE \hfil}

\hbox to 3.2in{\scriptsize\hbox to 0.9in{\pprno $A_5c$ \hfil} 0: AB, BC, CD; 1: ABC, AEF \hfil}

\hbox to 3.2in{\scriptsize\hbox to 0.9in{\pprno $D_5a$ \hfil} 0: BC, CD, DE, EF; 1: ABC \hfil}

\hbox to 3.2in{\scriptsize\hbox to 0.9in{\pprno $D_5b$ \hfil} 0: AB, CD, DE, EF; 1: ACD \hfil}

\hbox to 3.2in{\scriptsize\hbox to 0.9in{\pprno $D_5c$ \hfil} 0: AB, BC, CD, DE; 1: ABC \hfil}

\hbox to 3.2in{\scriptsize\hbox to 0.9in{\pprno $3A_2a$ \hfil} 0: AB, BC, DE, EF; 1: ABC, DEF \hfil ${\bf Z}_3$}

\hbox to 3.2in{\scriptsize\hbox to 0.9in{\pprno $3A_2b$ \hfil} 0: AB, CD, EF; 1: ABC, AEF, CDE \hfil ${\bf Z}_3$}

\hbox to 3.2in{\scriptsize\hbox to 0.9in{\pprno $A_1A_5a$ \hfil} 0: AB, BC, DE, EF; 1: ABC, ADE \hfil ${\bf Z}_2$}

\hbox to 3.2in{\scriptsize\hbox to 0.9in{\pprno $A_1A_5b$ \hfil} 0: AB, BC, CF, DE; 1: ABC, ADE \hfil ${\bf Z}_2$}

\hbox to 3.2in{\scriptsize\hbox to 0.9in{\pprno $A_1A_5c$ \hfil} 0: AB, BC, CD, DE, EF; 2: ABCDEF \hfil ${\bf Z}_2$}

\hbox to 3.2in{\scriptsize\hbox to 0.9in{\pprno $E_6$ \hfil} 0: AB, BC, CD, DE, EF; 1: ABC \hfil}
\vfil}}
\vskip.1in
\caption[\hskip.1in Configuration Types]{Configuration Types}\label{configtable}
\end{table}
}

\vfil\eject

As a check on our list of
configuration types as given in Table \ref{configtable}, we have the following
well known result, Theorem \ref{dynkthm}. 
(See \cite{refBW} for a version of the result in characteristic 0,
or see Theorem IV.1 of the arXiv version math.AG/0506611 of the
paper \cite{refGH} for a proof in general.
The proof is to study the morphisms $X\to\pr^2$ obtained by mapping $X$ to $\pr^3$
using the linear system $|-K_X|$, and then mapping the image $\bar X $
to $\pr^2$ by projecting from a singular point.)
Thus we get the same
Dynkin diagrams from Theorem \ref{dynkthm} as we found by a brute force determination
of configuration types. Moreover, one can (as we did in fact do)
find all exceptional configurations for each
of the 20 graphs listed in Theorem \ref{dynkthm}, and for each
exceptional configuration one can write down
the corresponding (representable) configuration type. Since by Proposition \ref{rprsntblty} every type is
representable over every algebraically closed field, it follows that the types obtained this way
should be (and in fact are) the same types we found by brute force.

\begin{thm}\label{dynkthm}
Let $X$ be a blow up of $\pr^2$ at 6 essentially
distinct points of $\pr^2$, such that $-K_X$ is nef. Assume that
$X$ has at least one $(-2)$-curve. Then the intersection graph of the
set of $(-2)$-curves is one of the following
20 graphs:
$A_1$,
$2A_1$,
$A_2$,
$3A_1$,
$A_1A_2$,
$A_3$,
$4A_1$,
$2A_1A_2$,
$A_1A_3$,
$2A_2$,
$A_4$,
$D_4$,
$A_12A_2$,
$2A_1A_3$,
$A_1A_4$,
$A_5$,
$D_5$,
$3A_2$,
$A_1A_5$, and
$E_6$. Each of these graphs occurs
as the graph of the set of $(-2)$-curves on some $X$,
and in a unique way (unique in the sense that
if the same graph occurs on two surfaces $X$ and
$X'$, then there are exceptional configurations
$L,E_1,\ldots,E_6$ on $X$ and $L',E_1',\ldots,E_6'$
on $X'$, such that a class $a_0L+\sum_ia_iE_i$
is the class of a $(-2)$-curve on $X$ if and only if
$a_0L'+\sum_ia_iE_i'$ is the class of a $(-2)$-curve on $X'$).
\end{thm}

\section{Resolutions}\label{resols}

Let $p_1,\ldots,p_6$ be essentially distinct points of $\pr^2$.
Let $Z=m_1p_1+\cdots+m_6p_6$ be a fat point subscheme of
$\pr^2$, and let $F(Z,i)=iL-m_1E_1-\cdots-m_6E_6$ on the surface
$X$ obtained by blowing up the points $p_i$.
As explained in Section \ref{bkgd}, the ideal
$I(Z)$ is obtained as follows. Let $\pi:X\to\pr^2$ be the
morphism to $\pr^2$ given by
blowing up the points $p_i$, and let
$L,E_1,\ldots, E_6$ be the corresponding
exceptional configuration. Let $F=-(m_1E_1+\cdots+m_6E_6)$.
Then $\C I_Z=\pi_*(\C O_X(-m_1E_1-\cdots-m_6E_6))$ is a sheaf
of ideals on $\pr^2$, and $I(Z)=\oplus_{i\ge0}H^0(\pr^2,\C I_Z\otimes\C O_{\pr^2}(i))$.
Also, we may as well assume that the coefficients $m_i$ satisfy the
proximity inequalities.
If they do not, there is another choice
of coefficients $m_i'$ which do satisfy them, giving a 0-cycle $Z'$
for which $I(Z)=I(Z')$. (The proximity inequalities are precisely
the conditions on the $m_i$ given by the inequalities
$F\cdot C\ge 0$ for each divisor class $C$ which is
the class of a component of the curves whose classes are
$E_1,\ldots,E_6$. In the case that the points $p_i$ are distinct,
the proximity inequalities are merely that $m_i\ge0$ for all $i$.
If $p_2$ is infinitely near $p_1$, then we would have the additional
requirement that $m_1\ge m_2$. This corresponds to the fact that
a form cannot vanish at $p_2$ without already vanishing at $p_1$.
If the $m_i$ do not satisfy the proximity inequalities, then
$F(Z,i)$ will never be nef: no matter how large  $i$ is, some
component $C$ of some $E_j$, $j>0$, will have
$F(Z,i)\cdot C<0$. Thus $C$ will be a fixed component of
$|F(Z,i)|$ for all $i$. By subtracting off such fixed components
one obtains a class $iL-(m_1'E_1+\cdots+m_6'E_6)$,
which also gives a 0-cycle $Z'=m_1'p_1+\cdots+m_6'p_6$
satisfying the proximity inequalities and which gives the same ideal
$I(Z)=I(Z')$. See \cite{refAppendix} for more details.)

The minimal free resolution of $I(Z)$ is an exact sequence of the form
$$0\to F_1\to F_0\to I(Z)\to 0$$
where each $F_i$ is a free graded $R$-module, with respect to the
usual grading of $R$ by degree, and all nonzero entries of the
matrix defining the homomorphism $F_1\to F_0$ are homogeneous
polynomials in $R$ of degree at least 1. Since $F_0$ and $F_1$ are
free graded $R$-modules, we know that there are integers $g_i$ and
$s_j$ such that $F_0\cong\oplus_i R[-i]^{g_i}$ and
$F_1\cong\oplus_j R[-j]^{s_j}$. These integers are the graded
Betti numbers of $I(Z)$. To determine the modules $F_1$ and $F_0$
up to graded isomorphism (or, equivalently, to determine the
graded Betti numbers of the minimal free resolution of $I(Z)$), it
is enough, as for distinct points (as explained in \cite{refGH}),
to determine $h^0(X,F(Z,i))$ and the ranks for all $i\ge0$ of the
multiplication maps $\mu_{Z,i} : I(Z)_i\otimes R_1\to I(Z)_{i+1}$
for each $i\ge0$, where, given a graded $R$-module $M$, $M_t$
denotes the graded component of degree $t$. Since (see
\cite{refGH}) the rank of $\mu_{Z,i}$ is the same as the rank of
$\mu_{F(Z,i)}: H^0(X,F(Z,i))\otimes H^0(X,L)\to H^0(X,L+F(Z,i))$,
it is enough to determine the rank of $\mu_{F(Z,i)}$.

As explained in \cite{refGH},
we can compute $h^0(X,F(Z,i))$ if we know $\hbox{NEG}(X)$
(or therefore even just $\hbox{neg}(X)$),
and we can compute the rank of $\mu_{F(Z,i)}$
if we can compute the rank of $\mu_F$ whenever $F$ is nef
(which the main result of this section, Theorem \ref{essdistpntsThm},
says we can do).

The method we use to prove Theorem \ref{essdistpntsThm} is precisely the method used in
\cite{refGH}. It involves the quantities
$q(F)=h^{0}({X},{F-E_1})$ and $l(F)=h^{0}({X},{F-(L-E_1})),$ and
bounds on the dimension of the cokernel of
$\mu_F$, defined in terms of quantities
$q^*(F)=h^{1}({X},{F-E_1})$ and $l^*(F)=h^{1}({X},{F-(L-E_1)})$,
introduced in \cite{refIGC} and \cite{refFHH}. A version of Lemma \ref{IGClem}
for distinct points is given in \cite{refIGC} and \cite{refFHH}, but
with only trivial changes the proof
for essentially distinct points is the same.

\begin{lem}\label{IGClem}
Let $X$ be obtained by blowing up
essentially distinct points $p_i\in\pr^2$, and let $F$ be the
class of an effective divisor on $X$ with $h^1(X,F)=0$. Then
$l(F)\le\hbox{dim ker $\mu_F$}\le q(F)+l(F)$ and $\hbox{dim cok
$\mu_F$}\le q^*(F)+l^*(F)$.
\end{lem}

\begin{rem}\label{IGCrem}
The quantities $q(F)$ and $l(F)$ are defined in terms of
$E_1$ and $L-E_1$, but in fact $E_j$, $j>0$, can
often be used in place of $j=1$. This is always true
if the points $p_i$ are distinct, since one can reindex the points.
Likewise, if the points are only essentially distinct,
any $j$ can be used so long as $p_j$ is a point on $\pr^2$, and not
only infinitely near a point of $\pr^2$.
\end{rem}

\begin{thm}\label{essdistpntsThm}
Let $X$ be obtained by blowing up
6 essentially distinct points of $\pr^2$. Let $L,E_1,\ldots,E_6$ be the
corresponding exceptional configuration. Assume that $-K_X$ is nef,
and let $F$ be a nef divisor.
Then $\mu_F$ has maximal rank.
\end{thm}

\begin{proof} The case of general points (i.e., that $\hbox{neg}(X)$ is empty)
is done in \cite{refFi} (but it can be recovered by the methods we use here).
This handles one of the 90 cases of Table \ref{configtable}.
Also, for 28 of the cases of Table \ref{configtable}, a conic goes through
the six points (i.e., $h^0(X, 2L-E_1-\cdots-2E_6)>0$); these cases are configuration
types 4, 8, 12, 16, 25, 26, 30, 33, 37, 42, 47, 48, 50, 53,
58, 61, 64, 66, 70, 72, 76, 78, 81, 83, 85, 88, 89 and 90.
The result holds for these cases by Theorem 3.1.2 of \cite{refFreeRes}
(also see Lemma 2.11 of \cite{refGH}).

Four of the remaining 61 cases correspond to distinct points,
and were handled in \cite{refGH}. These cases are 3, 9, 17 and 34.
The remaining cases are handled by the same method as these four.
The basic idea is this. If $F$ is a nef divisor such that $l(F)>0$, $q(F)>0$,
and $l^*(F)=0=q^*(F)$, then not only is it true that $\mu_F$ is surjective
(by Lemma \ref{IGClem}), but $l(F+G)>0$ and $q(F+G)>0$ by Lemma \ref{NEGisfinite},
and $l^*(F+G)=0=q^*(F+G)$ holds for all nef $G$
(by the proof of Corollary 2.8 of \cite{refGH}), hence
$\mu_{F+G}$ is surjective for all nef $G$.

Using Lemma 2.5 of \cite{refGH} one can easily give an explicit
list of generators of the nef cone for each configuration type. In
the best of all worlds, what would happen is that we would find
that $l(F)>0$, $q(F)>0$, $l^*(F)=0=q^*(F)$, for every $F$ in our
set of generators, and the result would be proved. But our world
is not the best of all imaginable worlds, so some additional work
is needed. In \cite{refGH} this is done, applying Corollary 2.8,
Lemma 2.9 and Lemma 2.10 of \cite{refGH}. These are all stated for
6 distinct points or $\pr^2$, but it is easy to check that the
proofs continue to hold for 6 essentially distinct points if
$-K_X$ is nef.

We now describe what this additional work is.
Let $\Gamma(X)$ be a set of generators
of the nef cone for $X$. (For practical purposes
of actually carrying out the calculations, it is best
to choose a minimal set of generators.)
Let $\Gamma_i(X)$ be the set of
all sums with exactly $i$ terms, where each term
is an element (with coefficient 1) of $\Gamma(X)$. Let
$S(X)$ be the set of all nef classes $F$ such that either
$q(F)=0$, $l(F)=0$ or $l^*(F)+q^*(F)>0$.
Then let $S_i(X)=S(X)\cap\Gamma_i(X)$; by
Corollary 2.8 \cite{refGH}, we have $S_{i+1}(X)\subseteq S_i(X)+S_1(X)$.
Typically the subsets $S_i(X)$ are nonempty.
But for $i\ge 3$, it always turns out that Lemma 2.9 \cite{refGH}
applies. This lemma involves a parameter $k$
which we can always take to be $k=2$. It also involves
a particular choice of class $C_F\in S_1(X)$ for each
$F\in S_i(X)$. The result is
that $S_i(X)\subseteq\{F+(i-3)C_F: F\in S_3(X), C_F\in S_1(X)\}$.

First one verifies directly that maximal rank holds
for $\mu_F$ for all $F\in S_i(X)$ for $i\le 3$, using Lemma \ref{IGClem}.
An induction (applying Lemma 2.10 \cite{refGH})
then verifies maximal rank for the strings $F+(i-3)C_F$, and hence
for all nef $F$.
There is one case that must be handled ad hoc (as was done by
\cite{refFi} and as we demonstrate below).
If $F=5L-2(E_1+\cdots+E_6)$, then $C_F=F$,
but $l(iF)>0$ for $i\ge3$ (so $\mu_{iF}$ is not injective by Lemma \ref{IGClem})
while $l^*(iF)>0$ for all $i$ (so the bounds
in Lemma \ref{IGClem} never force surjectivity). We now treat one case
in detail, as an example. The remaining cases are similar.

Consider configuration type 2, so $\hbox{neg}(X)=\{N\}$,
where $N=E_1-E_2$. Then $S_1(X)$ has 58
elements, $S_2(X)$ has 140, and $S_3(X)$, $S_4(X)$ and $S_5(X)$ have
150. Moreover, $\mu_H$ has maximal rank (by a case by case application of Lemma \ref{IGClem}) for
each element $H$ of $S_i(X)$, $1\le i\le 5$, except possibly
$mH$ when $H=5L-2(E_1+\cdots+E_6)$ for $m>1$ (since
$q(mH)+l(mH)>0$ and $q^*(mH)+l^*(mH)>0$ in these cases).
To show $\mu_H$ is onto for $H=2(5L-2(E_1+\cdots+E_6))$,
let $C=2L-E_1-\cdots-E_5$, and consider $F=H-C$.
Then $\mu_F$ is onto (by Lemma \ref{IGClem}, since $q^*(F)+l^*(F)=0$)
hence $\mu_H$ is onto (by Lemma 2.10 \cite{refGH}), and now
$\mu_{H+iC}$ is onto for all $i\ge0$ (also by Lemma 2.10 \cite{refGH}, taking
$F$ to be $mH$ and $C=5E_0-2(E_1+\cdots+E_6)$ for the induction in Lemma 2.10 \cite{refGH}).
By brute force check, applying Lemma 2.9 \cite{refGH} (with $k=2$ and $j=2$)
and Lemma 2.10 \cite{refGH}, it follows that $\mu_F$ has maximal rank
for every $F$ in each $S_i(X)$.
\end{proof}

\section{Examples}\label{exmpls}

Given only the configuration type and multiplicities
$m_1,\ldots,m_6$ satisfying the proximity inequalities,
Lemma \ref{NEGisfinite} and Theorem \ref{essdistpntsThm} allow us to determine
the Hilbert function and graded Betti numbers for $I(Z)$
for any fat point subscheme $Z=m_1p_1+\cdots+m_6p_6$
supported at 6 essentially distinct points $p_i$ of $\pr^2$
which when blown up give a surface $X$ for which $-K_X$ is nef (i.e., give a surface
isomorphic to the desingularization of a normal cubic surface).

The procedure for doing so is exactly the same as described
in detail in \cite{refGH}. We briefly recall the procedure.
Given $Z=m_1p_1+\cdots+m_6p_6$, to determine
$h_{I(Z)}(t)$, compute $h^0(X, F(Z,t))$,
where $F(Z,t)= tL-m_1E_1-\cdots-m_6E_6$.
To do this, let $D=F(Z,t)$
and check $D\cdot C$ for all prime divisors $C$ with $C^2<0$.
(Knowing the configuration type tells us the list of these
divisors $C$.) Whenever $D\cdot C<0$, replace
$D$ by $D-C$ and again check $D\cdot C$
with this new $D$ against all $C$. Eventually either
$D\cdot L<0$ (in which case $h^0(X, F(Z,t))=h^0(X, D)=0$),
or $D\cdot C\ge0$ for all $C$ (in which case,
by Lemma \ref{NEGisfinite}, $D$ is nef and $h^0(X, F(Z,t))=h^0(X, D)=(D^2-K_X\cdot D)/2+1)$.

To determine the graded Betti numbers, note that it suffices to compute
the Betti numbers $g_i$ for all $i$, since the exact sequence
$0\to F_1\to F_0\to I(Z)\to 0$ allows one to determine
$F_1$ up to graded isomorphism if one knows the graded Betti numbers
for $F_0$ and also the Hilbert function for $I(Z)$. To determine
$g_{t+1}$, note that $g_{t+1}=h^0(X, F(Z,t+1))$ if $h^0(X, F(Z,t))=0$.
If $h^0(X, F(Z,t))>0$, obtain the nef divisor $D$ from $F(Z,t)$ as above.
Then $g_{t+1}=(h^0(X, F(Z,t+1))-h^0(X,D+L))+\hbox{max}(0, h^0(X,D+L)-3h^0(X,D))$.

The procedure thus involves nothing more than taking dot products of integer vectors,
and can easily be done by hand. An awk script which automates the steps is available
at

\noindent \url{http://www.math.unl.edu/~bharbourne1/6ptsNef-K/Res6pointNEF-K}.

\noindent It can be run over the web at 

\noindent \url{http://www.math.unl.edu/~bharbourne1/6ptsNef-K/6reswebsite.html}.

Using this script we determined all possible Hilbert functions and graded Betti numbers for
fat points of the form $Z=p_1+\cdots+p_6$ and $2Z=2p_1+\cdots+2p_6$ for essentially distinct points
$p_i$ such that $-K_X$ is nef on the resulting surface $X$. For 6 essentially distinct
points with nef $-K_X$, this completely answers the questions raised
in \cite{refGMS}. We show what happens in Table \ref{ResHilbtable}. (The table regards 
a Hilbert function $h=h_{R/I(mZ)}$ as the sequence $h(0), h(1), h(2), \ldots$.
But any such $h$ reaches a maximum value at the regularity; i.e.,
for all $t$ greater than or equal to the regularity of $I(mZ)$, we have $h(t)=h(t+1)$.
Thus Table \ref{ResHilbtable} gives $h$ only up to this maximum value.)

In Table \ref{ResHilbtable}, case 1 occurs for the following configuration
types (as denoted in Table \ref{configtable}): 4, 8, 12, 16, 25, 26, 30, 33, 37, 42, 47,
48, 50, 53, 58, 61, 64, 66, 70, 72, 76, 78, 81, 83, 85, 88, 89, 90.
For each of these types, only one Hilbert function occurs for $2Z=2p_1+\cdots+2p_6$,
the one given as $1(a)$. These all have the same graded Betti numbers too.

Case 2 occurs for the remaining configuration
types: 1, 2, 3, 5, 6, 7, 9, 10, 11, 13, 14, 15, 17, 18, 19, 20, 21, 22, 23,
24, 27, 28, 29, 31, 32, 34, 35, 36, 38, 39, 40, 41, 43, 44, 45, 46, 49,
51, 52, 54, 55, 56, 57, 59, 60, 62, 63, 65, 67, 68, 69, 71, 73, 74, 75, 77, 79, 80, 82, 84, 86, 87.
For these, two different Hilbert functions occur for $2Z=2p_1+\cdots+2p_6$,
given as 2(a) and 2(b). Case 2(a) occurs for types 34, 68 and 87, and
these three all have the same graded Betti numbers.
Case 2(b) occurs for types 1, 2, 3, 5, 6, 7, 9, 10, 11, 13, 14, 15, 17, 18, 19,
20, 21, 22, 23, 24, 27, 28, 29, 31, 32, 35, 36, 38, 39, 40, 41, 43, 44, 45,
46, 49, 51, 52, 54, 55, 56, 57, 59, 60, 62, 63, 65, 67, 69, 71, 73, 74, 75, 77,
79, 80, 82, 84, and 86. These all have the same Hilbert function, but three different possibilities
occur for the graded Betti numbers, which we distinguish in the table by cases 2(b1), 2(b2) and 2(b3).
Case 2(b1) occurs for types 1, 2, 3, 5, 6, 7, 10, 11, 13, 14, 18, 19, 20, 21,
22, 27, 28, 31, 35, 36, 38, 43, 44, 49, 51, 54, 57, 59, 62, 73, 74 and 79.
Case 2(b2) occurs for types 9, 15, 23, 24, 29, 32, 39, 40, 46, 52, 55, 56, 60, 63, 67, 69, 71, 82 and 84,
and case 2(b3) occurs for  the remaining types
17, 41, 45, 65, 75, 77, 80 and 86.

\def\mytmprow#1#2#3#4{\hbox to\hsize{\hskip.2in\hbox to6.3in{\hbox to5.7in{\hbox to4.2in{\hbox to 1in{#1\hfil}#2\hfil}#3\hfil}#4\hfil}\hss}\par}
\vskip\baselineskip
\setcounter{table}{0}

\begin{table}
\hbox to\hsize{\hrulefill}
\mytmprow{{\bf Scheme}}{{\bf Resolution}}{{\bf Hilbert Function}}{}
\hbox to\hsize{\hrulefill}
\mytmprow{}{\hbox to1.3in{$F_1$\hfil} $F_0$}{$h_{R/I(mZ)}$, $m=1,2$}{}
%\hbox to\hsize{\leaders\hrule height 2pt\hfill}
\hbox to\hsize{\hrulefill}
\mytmprow{\hbox to .8in{\hbox to.2in{1:\hfil}\hfil$Z$}}{\hbox to1.3in{$R[-5]$\hfil} $R[-3]\oplus R[-2]$}{1, 3, 5, 6}{}
\mytmprow{\hbox to .8in{\hfil(a):\ \ $2Z$}}{\hbox to1.3in{$R[-8]\oplus R[-7]$\hfil} $R[-6]\oplus R[-5]\oplus R[-4]$}{1, 3, 6, 10, 14, 17, 18}{}
\hbox to\hsize{\hrulefill}
\mytmprow{\hbox to .8in{\hbox to.2in{2:\hfil}\hfil$Z$}}{\hbox to1.3in{$R[-4]^3$\hfil} $R[-3]^4$}{1, 3, 6}{}
\mytmprow{\hbox to .8in{\hfil(a):\ \ $2Z$}}{\hbox to1.3in{$R[-7]^4$\hfil} $R[-6]^4\oplus R[-4]$}{1, 3, 6, 10, 14, 18}{}
\mytmprow{\hbox to .8in{\hfil(b1):\ \ $2Z$}}{\hbox to1.3in{$R[-7]^3$\hfil} $R[-6]^1\oplus R[-5]^3$}{1, 3, 6, 10, 15, 18}{}
\mytmprow{\hbox to .8in{\hfil(b2):\ \ $2Z$}}{\hbox to1.3in{$R[-7]^3\oplus R[-6]$\hfil} $R[-6]^2\oplus R[-5]^3$}{1, 3, 6, 10, 15, 18}{}
\mytmprow{\hbox to .8in{\hfil(b3):\ \ $2Z$}}{\hbox to1.3in{$R[-7]^3\oplus R[-6]^2$\hfil} $R[-6]^3\oplus R[-5]^3$}{1, 3, 6, 10, 15, 18}{}
\hbox to\hsize{\hrulefill}
\vskip.05in
\caption[\hskip.1in Resolutions and Hilbert Functions]{Resolutions and Hilbert Functions}\label{ResHilbtable}
\end{table}

We close with one final example. The Hilbert functions that occur for $Z$ or $2Z$  for every choice
of 6 essentially distinct points $Z=p_1+\cdots+p_6\subset\pr^2$ all already occur
for distinct points. The first case of a Hilbert function that occurs for 6 essentially
distinct points $mZ$ of multiplicity $m$ that does not occur for any 6 distinct points of multiplicity $m$
is for $m=3$, and in this case there is only one, this being the Hilbert function for the ideal $I(Z)$ of 6 essentially
distinct points of multiplicity $3$ with configuration type 86, which is
$h_{I(Z)}(t)=0$ for $t<6$, $h_{I(Z)}(6)=1$, $h_{I(Z)}(7)=3$, and, 
for $t>7$, $h_{I(Z)}(t)=\binom{t+2}{2}-36$.
Applying the results of \cite{refGH}, we see
this Hilbert function does not occur for any configuration of 6 distinct points.
The graded Betti numbers for $I(Z)$ are such that $F_0\cong R[-9]^3\oplus R[-8]^3\oplus R[-6]$
and $F_1\cong R[-10]^3\oplus R[-9]^3$.

\end{document}